\documentclass{amsart}

\usepackage[pdfauthor={Montserrat Alsina, Dimitrios Chatzakos},
pdftitle={},
            pdfkeywords={Hyperbolic lattice points},
             pdfcreator={Pdflatex}]
{hyperref}
\hypersetup{colorlinks=true}
\usepackage{xcolor}
\usepackage{enumerate}

\usepackage{amssymb,amsmath,latexsym,amsthm}%or any usual package
\usepackage[english]{babel}

\newtheorem{theorem}{Theorem}[section]

\newtheorem{proposition}[theorem]{Proposition}

\theoremstyle{definition}

\numberwithin{equation}{section}
\DeclareMathOperator*{\vol}{vol}

\DeclareMathOperator{\Sym}{Sym}
\begin{document}

\newtheorem{remark}[theorem]{Remark}

%\renewcommand{\labelenumi}{(\roman{enumi})} 
%\def\theenumi{\roman{enumi}} 

%\numberwithin{equation}{section}

\def \g {{\gamma}}
\def \G {{\Gamma}}
\def \l {{\lambda}}
\def \a {{\alpha}}
\def \b {{\beta}}
\def \f {{\phi}}
\def \r {{\rho}}
\def \R {{\mathbb R}}
\def \H {{\mathbb H}}
\def \N {{\mathbb N}}
\def \C {{\mathbb C}}
\def \Z {{\mathbb Z}}
\def \F {{\Phi}}
\def \Q {{\mathbb Q}}
\def \e {{\epsilon }}
\def \ev {{\vec\epsilon}}
\def \ov {{\vec{0}}}
\def \GinfmodG {{\Gamma_{\!\!\infty}\!\!\setminus\!\Gamma}}
\def \GmodH {{\Gamma\backslash\H}}
\def \sl  {{\hbox{SL}_2( {\mathbb R})} }
\def \psl  {{\hbox{PSL}_2( {\mathbb R})} }
\def \slz  {{\hbox{SL}_2( {\mathbb Z})} }
\def \pslz  {{\hbox{PSL}_2( {\mathbb Z})} }
\def \L  {{\hbox{L}^2}}

\newcommand{\norm}[1]{\left\lVert #1 \right\rVert}
\newcommand{\abs}[1]{\left\lvert #1 \right\rvert}
\newcommand{\modsym}[2]{\left \langle #1,#2 \right\rangle}
\newcommand{\inprod}[2]{\left \langle #1,#2 \right\rangle}
\newcommand{\Nz}[1]{\left\lVert #1 \right\rVert_z}
\newcommand{\tr}[1]{\operatorname{tr}\left( #1 \right)}

\title[Hyperbolic lattice counting on compact surfaces]{CM-points and lattice counting on arithmetic \\ compact Riemann surfaces}
\author{Montserrat Alsina and Dimitrios Chatzakos}
\address{ Departament de Matem\`atiques, Escola Polit\`ecnica Superior d'Enginyeria de Manresa, Universitat Polit\`ecnica de Catalunya}
\email{montserrat.alsina@upc.edu}
\address{Universit\'e de Lille 1 Sciences et Technologies and Centre Européen pour les Mathématiques, la Physique et leurs interactions (CEMPI), Cit\'e Scientifique, 59655 Villeneuve d’ Ascq C\'edex, France}
\email{dimitrios.chatzakos@univ-lille.fr}
%\thanks{The second author was supported by a }
\date{\today}
\subjclass[2010]{Primary 11F72; Secondary 37C35, 37D40}

\begin{abstract}
 Let $\mathcal{X}(D,1) =\G(D,1) \backslash \mathbb{H}$ denote the Shimura curve of level $N=1$ arising from an indefinite quaternion algebra of fixed discriminant $D$. We study the discrete average of the error term in the hyperbolic circle problem over Heegner points of discriminant $d <0$ on $\mathcal{X}(D,1)$ as $d \to -\infty$. We prove that if $|d|$ is sufficiently large compared to the radius $r \approx \log X$ of the circle, we can improve on the classical $O(X^{2/3})$-bound of Selberg. Our result extends the result of Petridis and Risager for the modular surface to arithmetic compact Riemann surfaces. 
\end{abstract}
\maketitle
\section{Introduction}\label{Introduction}

 Let $\mathbb{H}$ denote the hyperbolic plane and $\rho(z,w)$ the hyperbolic distance of two fixed points $z,w \in \mathbb{H}$. For $\Gamma$ a cocompact or a general cofinite Fuchsian group, the hyperbolic lattice counting problem asks to estimate the quantity
\begin{equation}
N(X; z,w)= \# \left\{ \gamma \in \Gamma : \rho( z, \gamma w) \leq \cosh^{-1} \left(\frac{X}{2}\right) \right\},
\end{equation}
as $X \to \infty$. The study of this problem was initiated in \cite{delsarte}. Further, the asymptotic behaviour of $N(X; z,w)$ was studied in \cite{good, gunther, huber1, huber2, patterson, selberg} using tools from the spectral theory of automorphic forms in $\L(\GmodH)$. Let $\Delta$ be the Laplacian of the hyperbolic surface $\GmodH$ and $ \{u_j \}_{j=0}^{\infty}$ be an $\L$-normalized sequence of eigenfunctions of $-\Delta$ (Maass forms) with eigenvalues $ \{\lambda_j \}_{j=0}^{\infty}$. We write $\lambda_j= s_j(1-s_j) =1/4 +t_j^2$. The asymptotics of $N(X; z,w)$ can be deduced as an application of the spectral theorem for $\L(\GmodH)$ for appropriately chosen automorphic kernels (also called the pre-trace formula). Selberg \cite{selberg} first proved that, as $X \to \infty$,
\begin{equation} \label{mainformulaclassical}
 N(X;z,w)  \sim  M(X; z,w) := \sum_{1/2 < s_j \leq 1} \sqrt{\pi} \frac{\Gamma(s_j - 1/2)}{\Gamma(s_j + 1)} u_j(z) \overline{u_j(w)} X^{s_j}
\end{equation}
and
\begin{equation} \label{errorselbergbound}
E(X;z,w):= N(X;z,w)  -   M(X; z,w)  = O(X^{2/3}).
\end{equation}
For the error term $E(X;z,w)$ we have the conjecture
\begin{equation} \label{conjecture1}
E(X; z,w) = O_{\epsilon} (X^{1/2 + \epsilon})
\end{equation}
 for every $\epsilon > 0$ \cite{patterson, phirud}. We refer to \cite{cham2, cherubinirisager, hillparn} for upper bounds on the second moments of the error term and to \cite{chatz, phirud} for mean-value and $\Omega$-results supporting this conjecture.

More information can be obtained for the error term when the group $\Gamma$ is arithmetic. For $\G = \pslz$,
Petridis and Risager recently studied the average growth of $E(X;z,z)$ on compact subsets of the modular surface $\G \backslash \mathbb{H}$ both in the local \cite{petridisrisager1} and the discrete aspect \cite{petridisrisager2}. To state their discrete average result, for $d<0$ a fundamental discriminant denote by $\Lambda_d$ the set of CM (or Heegner) points of discriminant $d$ and by $h(d)$ the class number $h(d) = \# \Lambda_d$. If $f \geq 0$ is a smooth compactly supported function on $\G \backslash \mathbb{H}$, Duke's equidistribution theorem states that
\begin{eqnarray}
\frac{1}{h(d)} \sum_{ z \in \Lambda_d} f(z) \to \frac{1}{\vol(\G\backslash \mathbb{H})} \int_{\G \backslash \mathbb{H} } f(z) d \mu (z),
\end{eqnarray}
 as $d \to -\infty$. Applying Duke's theorem and their local average result for the error term over compact subsets of the modular surface \cite{petridisrisager1} they proved the following upper bound, which improves on average on Selberg's bound for $|d| \geq X^{11/2 + \epsilon}$.
\begin{theorem}[Petridis-Risager \cite{petridisrisager2}] \label{theorempetridisrisager}
Let $f$ be a non-negative smooth compactly supported function on the modular surface. Then, as $X \to \infty$, 
\begin{eqnarray}
\frac{1}{h(d)} \sum_{ z \in \Lambda_d} f(z) E(X;z,z) = O_{f, \epsilon} ( X^{7/12 +\epsilon} +  X^{4/5+\epsilon} |d|^{-4/165 + \epsilon}).
\end{eqnarray}
\end{theorem}
In this paper we study the discrete average of the error term $E(X;z,z)$ on the indefinite Shimura curves $\mathcal{X}(D, N) = \Gamma(D,N) \backslash \mathbb{H}$. Here $\G(D, N)$ is the cocompact discrete subgroup of $\sl $ of level $N$ arising from an indefinite quaternion algebra 
\begin{eqnarray*}
H= \left(\frac{a, b}{\mathbb{Q}}\right)
\end{eqnarray*}
of discriminant $D:=D_H$ (see subsection \ref{quaternionic}).  Let also $\Lambda_{D,N,d}$ denote the set of CM points on $\mathcal{X}(D, N)$ of fundamental discriminant $d<0$ (when non-empty) and $h(D, N, d) = \# \Lambda_{D,N,d}$ the class number counting the number of CM points on $\mathcal{X}(D, N)$.
%(\textcolor{red}{when this class number is nonzero and we average over those}) 
%Duke's equidistribution theorem in a quantitative form has been extended to this case by Harcos and Michel \cite{harcosmichel}.
For simplicity reasons, for the rest of the paper we restrict to level $N=1$ and we fix the notation $\Lambda_{D,d}=\Lambda_{D,1,d}$ and $h(D,d) := h(D,1,d)$. Applying Duke's equidistribution theorem for arihtmetic compact surfaces, we prove the following theorem.

\begin{theorem} \label{ourmaintheorem} 
Let $H$ be an indefinite quaternion algebra over $\mathbb{Q}$ of discriminant $D$ and $\G(D,1)$ the arithmetic cocompact Fuchsian group of level $1$ arising from $H$. Let $f$ be a non-negative smooth function on the Shimura curve $\mathcal{X}(D,1) = \Gamma(D,1) \backslash \mathbb{H}$. Then
\begin{eqnarray} \label{ourtheorembound} 
\frac{1}{h(D, d)} \sum_{ z \in \Lambda_{D, d}} f(z) E(X;z,z) = O_{f, D, \epsilon} (X^{5/8+\epsilon} + X^{4/5+\epsilon} |d|^{-2/85+\epsilon} )
\end{eqnarray}
as $X \to \infty$. 
\end{theorem}
Here, the limit is considered over $d$'s such that the class number is nonzero.
As a corollary, our result improves on Selberg's bound for arithmetic compact surfaces when the discriminant satisfies $|d| \geq X^{\frac{34}{5}+\epsilon}$. For $|d| \geq X^{\frac{119}{16}+\epsilon}$ we get an upper bound $O(X^{5/8+\epsilon})$. 

A main ingredient in the proof of Theorem \ref{ourmaintheorem} is the recent local average result of Biro \cite{biro2} for general Riemann surfaces of finite area. This is the first reason our upper bound is slightly weaker from that of Theorem \ref{theorempetridisrisager}. A second reason is we have not tried to optimize the exponents in all the possible aspects, for instance the bound (\ref{iwaniecsarnakbound}) on average, as our main goal is to indicate how the method of \cite{petridisrisager1} can be combined with the spectral result of Biro \cite{biro2} to deduce a bound for the discrete average of $E(X;z,z)$ (see section \ref{auxillaryresultssection} for more details).
In section \ref{auxillaryresultssection} we discuss the necessary background and the tools in the proof of Theorem \ref{ourmaintheorem} (equidistribution and bounds for Weyl sums, Biro's local average result, bounds for inner products of eigenfunctions). We give the proof of our result in section \ref{proofofthemainresult}.

\subsection{Acknowledgments}We are grateful to the School of Mathematics at the University of Bristol, where this work initially started, for the hospitality and the friendly enviroment. The first author is supported by Grant MTM2015-63829-P and Grant MTM2015-66716-P. The second author would also like to thank the Mathematics department of King's College London, for the support during the academic year 2016-17. This work was partially supported by the LMS 150th Anniversary Postdoctoral Mobility Grant 2016-17 and the European Union’s Seventh Framework Programme (FP7/2007-2013) / ERC grant agreements no 335141 Nodal. Currently, he is supported by a Labex CEMPI (ANR-11-LABX-0007-01). We would like to thank P. Bayer for useful comments, V. Blomer, Ph. Michel, P. Nelson, R.M. Nunes for useful comments on subconvexity and Y. Petridis for his valuable advice.

\section{Background and auxillary results} \label{auxillaryresultssection}

In this section we summarize the necessary background about the arithmetic compact Riemann surfaces $\mathcal{X}(D,N)$ and the spectral theory of automorphic forms on them. 

\subsection{Background on compact arithmetic groups} \label{quaternionic}

We refer to \cite[p.~301]{iwaniecsarnak}, \cite[p.~385]{koyama} as a basic reference for quaternionic groups and to \cite{alsina} for a detailed study of their arithmetic theory. For $a,b$ square-free integers with $a>0$ let 
\begin{eqnarray}
H = \left(\frac{a, b}{\mathbb{Q}}\right)
\end{eqnarray}
be the indefinite quaternion algebra over $\mathbb{Q}$ linearly generated by $\{1,\omega, \Omega, \omega \Omega \}$, where $\omega^2=a$, $\Omega^2=b$ and $\omega \Omega + \Omega \omega=0$. We say that $H$ ramifies at $p$ if $H_p = H \otimes_{\mathbb{Q}} \mathbb{Q}_p$ is a division algebra. Let $N(x)$ denote the norm of $x \in H$. For $R$ a maximal order of $H$, denote by $R(n)$ the elements $x\in R$ of norm $n$. Denote also by $\phi$ the embedding that maps 
\begin{eqnarray}
x = x_0 +x_1 \omega + x_2 \Omega + x_3 \omega \Omega = \xi + \eta \Omega
\end{eqnarray}
to 
\begin{eqnarray}
\phi (x) = \left( \begin{array}{cc}  \overline{\xi} &  \eta \\  b \overline{\eta} &  \xi \end{array}\right).
\end{eqnarray}
Then, the group
\begin{eqnarray}
\Gamma_H := \phi(R(1)) \subset \sl
\end{eqnarray}
is a Fuchsian group of the first kind, and the quotient $\Gamma_H \backslash \mathbb{H}$ is a Riemann surface. The discriminant $D = D_H$ of the algebra is defined as the product of the ramified primes of $H$. 
%For the rest of this paper, we assume that $H$ is unramified at $2$. In particular, the discriminant $D = D_H$, which is defined as the product of the ramified primes of $H$, is odd. 
 
Now let $D, N$ be natural numbers such that: i) $D$ is a product of an even number of different primes and, ii) $N >1$ is a natural number with $(D, N) = 1$. If $R=\mathcal{O}(D,N)$ is an Eichler order of level $N$ and $\phi$ a monomorphism $: H \to M_2(\mathbb{R})$, we denote the group $\Gamma_H$ associated to $\mathcal{O}(D,N)$ and $\phi$ by $\Gamma (D,N)$. We have 
\begin{eqnarray}
\Gamma (D,N) \subset {{\hbox{SL}_2( {\mathbb Q}(\sqrt{a})} }).
\end{eqnarray} 

The theory of Shimura provides a canonical model $\mathcal{X}(D, N)$ for $\Gamma (D,N) \backslash \mathbb{H}$ and a modular interpretation. The canonical model $\mathcal{X}(D, N)$ is a projective curve defined over $\mathbb{Q}$.  The quotient $\Gamma (D,N) \backslash \mathbb{H}$ is noncompact if $D=1$ (the nonramified case), and in this case we have $\mathcal{X}(1,N) = \Gamma_0 (N) \backslash \mathbb{H}$. If $D>1$, then $\mathcal{X}(D,N)$ is a compact surface. As we mentioned in section \ref{Introduction}, we restrict to the level one case. The curves $\mathcal{X}(D,1)$, for $D>1$, can be viewed as the compact analogues of the modular surface $\mathcal{X}(1,1) = \pslz \backslash \mathbb{H}$. The notion of CM points is now analogous to the case of the modular surface (see \cite[Def.~6.4]{alsina}) and we use the crucial fact that the class number $h(D, d)$ is always finite. 

\subsection{Hecke operators} When the group $\Gamma$ is arithmetic, there exists a large commutative algebra (Hecke algebra) of invariant, self-adjoint operators, the Hecke operators. They are defined by averaging over the orbits $R(1)\backslash R(n)$: for every $n \geq 1$, the $n$-th Hecke operator $T_n: \L (\mathcal{X}(D,1)) \to \L (\mathcal{X}(D,1))$ is defined by
\begin{eqnarray}
T_n f(z) = \sum_{x \in R(1)\backslash R(n)} f( \phi(x) z).
\end{eqnarray}
Since we are working in the level one case, for every $n \geq 1$ the Hecke operator $T_n$ is self-adjoint and commutes with the Laplacian $-\Delta$. Hence, we can choose a basis of eigenfunctions $u_j$ which simultaneously diagonalizes the Hecke operators and $-\Delta$. Such an eigenfunction is called a Hecke-Maass from. We always assume that $u_j$ are $\L$-normalized.

\subsection{Remarks on the explicit Jacquet-Langlands correspondence} Jacquet and Langlands proved that there exists a nice, bijective correspondence between the nontrivial automorphic forms on the multiplicative group of a division quaternion algebra and certain cusp forms on $ {{\hbox{GL}(2)}}$. The explicit Jacquet-Langlands correspondence was worked out by Hejhal \cite{hejhal}, Koyama \cite{koyama} (in view of an application to the Prime geodesic theorem), Bolte and Johansson \cite{bolte} and Str\"ombergsson \cite{strombergsson}, and  is given by the theta map (see \cite{strombergsson}). It was further investigated by Risager \cite{risager}, Arenas \cite{arenas} and Blackman and Lemurell \cite{blackman}.

We refer to the explicit Jacquet-Langlands correspondence in various places; it may be possible that one can improve on some of our estimates by referring further  to the explicit correspondence between $\G(D,1)$ and $\G_0(D)$. First, following the method of Iwaniec and Sarnak it may be possible be to improve bound (\ref{iwaniecsarnakbound}) on average. Secondly, Nelson \cite{nelson1} used the theta correspondence to study the quantum variance for cocompact groups. In principle, one may be able to use this to improve on Biro's result for $\G(D,N)$. However, Biro's result is more spectral in nature than arithmetic whereas the explicit study of the quantum variance in  \cite{nelson1} is quite arithmetic.

\subsection{Weyl's law and sup-norm estimates} \label{weylslaw}
We will apply estimates for the growth of Maass forms in the spectral limit. Weyl's law describes the asymptotics of the number of Laplacian eigenvalues up to a fixed height; more precisely for any cocompact Fuchsian group $\Gamma$ we have the asymptotic 
\begin{eqnarray}
\sum_{\lambda_j \leq T^2} 1 \sim \sum_{|t_j| \leq T} 1  \sim \frac{\vol (\GmodH)}{4 \pi} T^2,
\end{eqnarray}
see \cite[Corollary~11.2]{iwaniec}. The following estimate is the local version of Weyl's law for $\L(\GmodH)$.
\begin{theorem}[Local Weyl's law] \label{localweylslaw} 
For every $z$, as $T \to \infty$,
\begin{eqnarray*} \label{localweyl}
\sum_{|t_j| <T} |u_j(z)|^2  \sim c T^2,
\end{eqnarray*}
where the constant $c=c(z)$ depends only on the number of elements of $\Gamma$ fixing $z$.
\end{theorem}
See \cite[Lemma~2.3]{phirud} for a proof of this result. We emphasize that if $z$ remains in a compact set of $\mathbb{H}$ the constant $c(z)$ remains uniformly bounded. Thus, for compact surfaces we trivially have the uniform upper bound
\begin{eqnarray} \label{localweylslawcorollary} 
 \left\| \sum_{T \leq |t_j| \leq 2T} |u_j(z)|^{2} \right\|_{\infty} \ll T^{2}
\end{eqnarray}
This follows also from Bessel's inequality \cite[Proposition~7.2]{iwaniec} and the fact that the height function $y_{\G}(z)$ is uniformly bounded for $\G$ cocompact. On the other hand, for sup norm estimates of Hecke-Maass forms $u_j(z)$ on compact arithmetic surfaces we have the following deep subconvexity bound of Iwaniec and Sarnak \cite{iwaniecsarnak}
\begin{eqnarray} \label{iwaniecsarnakbound}
\|u_j(z)\|_{\infty}  \ll |t_j|^{5/12 + \epsilon}
\end{eqnarray}
as $\lambda_j \to \infty$. Conjecturally, the sup-norm of the Maass-Hecke cusp form is expected to be bounded by $t_j^{\epsilon}$. %Further, using the recent work for sup-norm estimates on congruence subgroups \cite{saha}, \cite{templier} and the Jacquet-Langlands correspondence one can deduce explicit dependence of the sup norm on both the spectral and the $D$ aspects.
Assuming that one would be able to control the dependence on $D$ (i.e. on $\Gamma$) in all the %other
 $O$-estimates for the local and the discrete average, it would be interesting to study the dependence of bound (\ref{ourtheorembound}) on $D \to \infty$.

\subsection{The equidistribution of CM points, Weyl sums and bounds for associated $L$-functions} \label{dukesresultforus}

Equidistribution of CM points has been extensively studied by various authors. Duke proved his celebrated result about Heegner points on the modular surface using estimates for the Fourier coefficients of Maass forms (see also \cite{michelvenkatesh}). Michel \cite{michel} proved the equidistribution result on small Galois orbits for the case of Shimura curves associated to definite quaternion algebras over $\mathbb{Q}$ and Harcos-Michel \cite{harcosmichel} worked the indefinite case. Zhang \cite{zhang} generalized the equidistribution result for CM points on Shimura varieties.  

\begin{theorem}{\cite[Theorem~1]{duke}, \cite[Theorem~9]{michelvenkatesh}}
Let $\Gamma = \Gamma(D,N)$ be the Fuchsian group as above and let $h(D,N,d)$ denote the number of CM points on $\mathcal{X}(D,N)$ of fundamental discriminant $d<0$. Then, for any compactly supported $f$ on $\mathcal{X}(D,N)$ we have as $d \to -\infty$,
\begin{eqnarray} \label{dukecocompact}
\frac{1}{h(D,N,d)} \sum_{ z \in \Lambda_d} f(z) \to \frac{1}{\vol(\mathcal{X}(D,N))} \int_{\mathcal{X}(D,N)} f(z) d \mu (z).
\end{eqnarray}
That means the CM points equidistribute on the Riemann surface $\mathcal{X}(D, N) $ with respect to the hyperbolic measure. 
\end{theorem}
To understand the discrete average of the error term for the lattice counting problem we will apply the pre-trace formula. This allows us to write the average of the error over the CM points in relation to the average of Maass forms over the CM points, the so-called Weyl sums:
\begin{eqnarray*}
 W(d, t_j) :=  \sum_{ z \in \Lambda_d} u_j(z).
\end{eqnarray*}
The effectivity in bound (\ref{dukecocompact}) is measured by estimates for the corresponding Weyl sums and the explicit Jacquet-Langlands correspondence.

We can bound the sum $W(d, t_j)$ referring to a Waldspurger-Zhang type formula. We refer to \cite[Lecture~2]{harcos}, \cite[Section~6]{harcosmichel}, \cite[Section~2]{michelvenkatesh} for extended comments on this formula. Here, we focus on a bound on the $d$-aspect. By \cite[Section~6]{harcosmichel} for $|t_j| \to \infty$ we deduce the bound 
\begin{eqnarray} \label{weylsumupperbound}
\frac{W(d, t_j)}{h(D, d)}   \ll |t_j|^{A} |d|^{-\frac{1}{28}},
\end{eqnarray}
The class number satisfies the classical bound $|d|^{1/2-\epsilon} \ll h(D,d) \ll |d|^{1/2+\epsilon}$ (see \cite[Proposition~3.55, Theorem~4.19]{alsina}).
The Weyl sum is related to central values of Maass forms through a general Waldspurger-Zhang formula in \cite[page~647]{harcosmichel}:
\begin{eqnarray} \label{waldspurgerzhang}
 |W(d, t_j)|^2 = c_{d, u_j} \frac{\sqrt{|d|} L(u_j \times \chi_d, 1/2)   L(u_j , 1/2) }{L(\Sym^2 u_j, 1)}
\end{eqnarray}
where $c_{d, u_j}\ll (1+ |d   t_j|)^{\epsilon}$. Convexity bounds for the $L$-factors imply $A=1/2$ in (\ref{weylsumupperbound}). Further, we can obtain the following estimate:
\begin{proposition} \label{propositionweylsums} For the Weyl sums we have the upper bound
\begin{eqnarray} \label{weyls sum improvement on average}
 \sum_{T \leq t_j \leq 2T} \frac{|W(d, t_j)|^2}{h(D, d)^2} \ll |d|^{-\frac{1}{6} +\epsilon} T^{2+\epsilon}.
\end{eqnarray}
\end{proposition}
Using Ivi\'c \cite[page~455]{ivic}, Young \cite[Theorem~1.1]{young}, Hoffstein-Lockhart \cite[Theorem~0.1]{hoffstein} and Jacquet-Langlands correspondence, the estimate follows applying the H\"older inequality, see \cite[Lemma~2.2]{petridisrisager2}.
The conjectural Lindel\"of estimate 
\begin{eqnarray*}
L(u_j \times \chi_d, 1/2) \ll (1+ |d   t_j|)^{\epsilon}
\end{eqnarray*}
implies the exponent $|d|^{-\frac{1}{2} +\epsilon}$ in (\ref{weyls sum improvement on average}).

\subsection{The local average of the error term for general cofinite groups} We now explain the main idea relating the discrete average of the error term $E(X;z,z)$ over CM points with the local average over a compact set. Roughly speaking, in the proof of the Theorem \ref{ourmaintheorem} we bound the discrete average of the error over the CM points by two pieces:
\begin{eqnarray} \label{errorsplit}
\frac{1}{h(D, d)} \sum_{ z \in \Lambda_d} f(z) E(X;z,z)  \ll E_{local} (X) + E (X),
\end{eqnarray}
where the $E_{local}$ is the average of the error over compact sets of $\GmodH$: 
\begin{eqnarray*}
E_{local}(X) := \int_{\GmodH} f(z) E(X;z,z) d \mu(z).
\end{eqnarray*}
To prove our result we need bounds for both summands of (\ref{errorsplit}), which we have to treat them separately. The tools we discussed in subsections \ref{weylslaw} and \ref{dukesresultforus} will be useful for dealing with the term $E (X)$, but the term $E_{local}(X)$ has to be controlled differently. 

In the case of the modular group $\G = \pslz$, the local and discrete averages were studied seperately in the independent works \cite{petridisrisager1}, \cite{petridisrisager2}. The treatment of the local average uses the classical result of Luo-Sarnak \cite{luosarnak} for the Quantum Unique Ergodicity of Maass cusp forms on average and can be applied for similar groups (arithmetic groups with cusps). Instead of using any strong arithmetic QUE information, Biro recently proved the following local average for general surfaces of finite area using a modified Selberg trace formula. For the case of compact arithmetic surfaces we apply his result.
\begin{theorem}{\cite[Theorem~1.1]{biro2}} \label{theorembiro}
Let $\G$ be a cofinite Fuchsian group and let $f$ be a non-negative smooth compactly supported function on $\GmodH$. Then, as $X \to \infty$, the local average of the error term satisfies the bound
\begin{eqnarray} \label{birobound}
\int_{\GmodH} f(z) E(X;z,z) d \mu(z) = O_{f, \Gamma, \epsilon} (X^{5/8+\epsilon}).
\end{eqnarray}
\end{theorem}

Biro was able to deduce his result by using his modification of the Selberg trace formula from \cite{biro}. In this work, instead of estimating the integral of the automorphic kernel $K(z,w)$ over the diagonal, Biro worked with the integral of the automorphic kernel against some automorphic function $u$ of eigenvalue $\lambda$. On this way, he deduced an identity between the geometric and the spectral side, where in the spectral side of the new trace formula he appeared the integrals 
\begin{eqnarray*}
\int_{\GmodH} |u_j(z)|^2 u(z)  d \mu(z).
\end{eqnarray*}
in the place of the usual $L^2$-norms of $u_j$ appearing in Selberg's formula. In the geometric side, he appeared weighted orbital integrals in the place of regular orbital integrals for the elliptic and hyperbolic terms (where the weight function depends on periods of $u$). The appearance of the period integrals of $u$ on the geometric side is the key ingredient that imply the improvement on the error term in (\ref{birobound}).

\subsection{Bounds for inner products of eigenfunctions} \label{boundsforinnerproducts}

We refer to the following uniform upper bound for the inner product of two eigenfunctions. The proof given in \cite[Theorem~5.1]{petridisrisager2} relies on bounds for weight $k$ eigenfunctions of $\Gamma$. Although it holds for a general cofinite group $\Gamma$, we state its simplified version only for our special case of a cocompact group. 

\begin{proposition} \label{propositionsuqresofeigenfunctions} For $f$ a smooth function on the compact surface $\GmodH$, $u_m, u_k$ two Maass forms and $a>0$ we have the following bound:
\begin{eqnarray}
\langle f |u_m|^2, u_k \rangle   \ll_{a,f} \left(  \frac{1+|t_m|}{1+|t_k|}   \right)^a    \|u_m\|_{\infty} \|u_m\|_2 \| u_k \|_2.
\end{eqnarray}
\end{proposition}

\section{Proof of the main theorem}   \label{proofofthemainresult} 

\subsection{Local average of a Maass form over CM points} 
Before proceeding to the proof of our main result, we start with a proposition for the average of  Maass forms multiplied with $f$ over Heegner points. Set 
\begin{eqnarray} \label{definitiongdt}
 G(d, t_j) = \frac{1}{h(D, d)} \sum_{ z \in \Lambda_d} f(z) |u_j(z)|^2 - \frac{1}{\vol(\mathcal{X}(D,1))} \int_{\mathcal{X}(D,1)} f(z) |u_j(z)|^2 d \mu(z).
\end{eqnarray}
We have the following estimate. 
\begin{proposition} \label{propositionboundsgdt}
Let $f$ be a smooth function on $\mathcal{X}(D,1)$ and $u_j(z)$ a Maass form on $\mathcal{X}(D,1)$ with eigenvalue $\lambda_j = 1/4+t_j^2 >1/4$. Then
\begin{eqnarray*}
G(d, t_j) =  O_{f, \Gamma, \epsilon} ( \|u_j \|_{\infty} |d|^{-1/12+\epsilon} |t_j|^{1+\epsilon}).
\end{eqnarray*}
\end{proposition}
\begin{proof}
Applying the spectral expansion to $F(z) = f (z) |u_j(z)|^2$ and averaging over the CM points we conclude
\begin{eqnarray*}
G(d, t_j) = \sum_{\lambda_k \leq 1/4} \langle F, u_k \rangle \frac{W(d, t_k)}{h(D,d)} + \sum_{t_k \in \mathbb{R}-\{0\}} \langle F, u_k \rangle  \frac{W(d, t_k)}{h(D,d)},
\end{eqnarray*}
where $\langle,\rangle$ denotes the Petersson inner product in $\L(\mathcal{X}(D,1))$. For a small eigenvalue $\lambda_k \leq 1/4$ we have
\begin{eqnarray*}
W(d, t_k) \ll_{\Gamma} h(D,d) \Longrightarrow \langle F, u_k \rangle \frac{W(d, t_k)}{h(D,d)} \ll_{\Gamma} \| f \|_{\infty}.
\end{eqnarray*}
We conclude the contribution of the small eigenvalues is $O_{f, \Gamma} (1)$. For the big eigenvalues $\lambda_k \geq 1/4$, Proposition \ref{propositionweylsums} and Cauchy-Schwarz in the interval $[T,2T]$ imply
\begin{eqnarray} \label{cauchyschwarzinner}
\sum_{T \leq t_k \leq 2T} \langle F, u_k \rangle  \frac{W(d, t_k)}{h(D,d)} \ll \left(\sum_{T \leq t_k \leq 2T} \left|\langle F, u_k \rangle \right|^2 \right)^{1/2} |d|^{-\frac{1}{12}+\epsilon} (1+T)^{1+\epsilon}.
\end{eqnarray}
Applying the estimate of Proposition \ref{propositionsuqresofeigenfunctions} for the sum over $t_k \in [T,2T]$, since all Maass forms are $\L$-normalized we get the left side of (\ref{cauchyschwarzinner}) is bounded by 
\begin{eqnarray*}
\ll_{a,f} \|u_j\|_{\infty}  \left( 1+|t_j|   \right)^{a}   |d|^{-\frac{1}{12}+\epsilon} (1+T)^{2-a+\epsilon}
\end{eqnarray*}
for any $a>0$. On the other hand, using Bessel's inequality we get        
\begin{eqnarray} \label{cauchyschwarzinner}
 \left(\sum_{T \leq t_k \leq 2T} \left|\langle f (z) |u_j(z)|^2, u_k \rangle \right|^2 \right)^{1/2}  \ll \| f (z) |u_j(z)|^2    \|_2 \ll_f   \| u_j\|_{\infty} \|u_j\|_{2} = \|u_j\|_{\infty}
\end{eqnarray}
uniformly for any $T$. Balancing between the two expressions, for any $A \gg 1$  sufficient large and fixed we get
\begin{eqnarray*} 
\sum_{ t_k  \in \mathbb{R}-\{0\}} \langle F, u_k \rangle  \frac{W(d, t_k)}{h(D,d)} \ll_{f, \Gamma, \epsilon}
\|u_j\|_{\infty}  |d|^{-\frac{1}{12}+\epsilon}   \left( A^{1+\epsilon}  +  \left( 1+|t_j|   \right)^{a}     A^{2-a + \epsilon}  \right).
\end{eqnarray*}
Picking $A \asymp t_j^{\frac{a}{a-1}}$ with $a$ large enough we deduce the statement.
\end{proof}

\subsection{Proof of Theorem \ref{ourmaintheorem}} 
We can now give the proof of our theorem. By a standard argument, we need to apply the pre-trace formula to suitable chosen kernels. We refer to \cite[Section~2]{cham2}, \cite[Section~6]{petridisrisager2} for more details on the calculations. For the proof we will refer to some standard estimates for the Selberg/Harish-Chandra transforms of the characteristic kernel $k(u) = \chi_{[0, (X-2)/4]}(u)$ and smooth approximations of $k(u)$ (see Proposition \ref{selberghasichandratransform}).

\begin{proof}
Define $R = \cosh^{-1} (X/2)$, denote by $u(z,w)$ the standard point-pair invariant function
\begin{eqnarray}
u(z,w) = \frac{|z-w|^2}{4 \Im(z) \Im(w)},
\end{eqnarray}
and let $k(u)$ be the kernel
\begin{eqnarray}
k(u) = \chi_{[0, (X-2)/4]}(u),
\end{eqnarray}
where by $\chi_{[A,B]}(u)$ we denote the characteristic function of the interval $[A,B]$. Recall that $\rho(z,w) \leq R$ if and only if $4u(z,w) \leq X-2$, and thus it is straightforward to see that the automorphic kernel of $k(u)$ satisfies 
\begin{eqnarray}
N(X,z,z) = K(z,z) := \sum_{\gamma \in \Gamma} k(u(\gamma z,z)).
\end{eqnarray}
The first problem arising in the study of the hyperbolic counting problem is that $k(u)$ is not an admissible kernel for the pre-trace formula; this can be checked using the asymptotic behaviour of its Selberg/Harish-Chandra transform $h(t)$ of $k$ (see \cite[eq.~(1.62)]{iwaniec} for the definition of the transform), and one has to approximate $k$ by some suitable test functions.

 For this reason, for $\delta>0$ small (that will be specified later), we define the smoothed kernel $k_{\delta}(u)$ by
\begin{eqnarray}
k_{\delta}(u) = \frac{1}{4 \pi \sinh^2 (\delta/2)} \chi_{[0, (\cosh \delta -1)/2]}(u).
\end{eqnarray}
We also define the hyperbolic convolution of two functions by
\begin{eqnarray}
k_1 \ast k_2 (u(z,w))= \int_{\mathbb{H}} k_1 (u(z,v)) k_2 (u(v,w)) d \mu(v),
\end{eqnarray}
and the kernels $k_{\pm}(u)$ defined by
\begin{eqnarray}
k_{\pm}(u) = \left(\chi_{[0,(\cosh (R\pm \delta)-1)/2]} \ast k_{\delta}\right) (u)
\end{eqnarray}
If $h_i(t)$ stands for the S/H-C transform of $k_i(u)$, then the S/H-C transform of the hyperbolic convolution $k_1 \ast k_2$ is given the product $h_1(t) h_2(t)$ (see \cite[pages~322-323]{cham2}). Using this observation we can compute the S/H-C transform of kernels $k_{\pm}(u)$.
\begin{proposition} \label{selberghasichandratransform} With notation as above, the S/H-C transforms $h_{\pm} (t)$ of the kernels $k_{\pm}(u)$ satisfy:
\\a) If $\lambda_j <1/4$, i.e. $s_j=1/2+it_j$ with $t_j \notin \mathbb{R}$ then
\begin{eqnarray}
h_{\pm} (t_j) = \sqrt{\pi} \frac{\Gamma(|t_j|)}{\Gamma(|t_j| + 3/2)} X^{1/2+|t_j|} + O (X \delta + X^{1/2}).
\end{eqnarray}
b)  If $\lambda_j = 1/4$, i.e. $t_j =0$ then
\begin{eqnarray}
h_{\pm} (0) = O (X^{1/2} \log X).
\end{eqnarray}
c) If $\lambda_j >1/4$, i.e. $t_j \in \mathbb{R}-\{0\}$ then
\begin{eqnarray}
h_{\pm} (t_j) = O \left( \frac{X^{1/2}}{|t_j|^{3/2}} \min \left\{1, \frac{1}{(\delta |t_j|)^{3/2}} \right\} \right).
\end{eqnarray}
For $t$ real we also have the uniform bound $h_{\pm} (t_j) = O \left( X^{1/2} \log X \right)$.
\end{proposition}
\begin{proof}
Part $a)$ follows from \cite[eq.(2.14), p.~323]{cham2} and part $b)$ from \cite[Lemma~2.4, (d)]{cham2}. Part $c)$ is calculated in \cite[eq.~(5.5), (5.10)]{petridisrisager1}, \cite[eq.~(6.3), (6.4)]{petridisrisager2}.
\end{proof}
Since $k_{-} \leq k \leq k_{+}$ \cite[page~15]{petridisrisager2}, summing over the group we deduce the inequalities
\begin{eqnarray*}
K^- (z,z) \leq N(X,z,z) \leq K^+ (z,z),
\end{eqnarray*}
hence
\begin{eqnarray}
\sum_{ z \in \Lambda_d} f(z) E(X;z,z) \ll \min_{\pm}\left|\sum_{ z \in \Lambda_d} f(z) \left( K^{\pm}(z,z) - M(X;z,z) \right)\right|.
\end{eqnarray}
The kernels $k_{\pm}$ are admissible in the pre-trace formula, which implies the expansions
\begin{eqnarray}
K^{\pm} (z,z)  = \sum_{t_j} h^{\pm} (t_j) |u_j(z)|^2.
\end{eqnarray}
Since we have a finite number of eigenvalues $0 \leq \lambda_j \leq 1/4$, we conclude
\begin{eqnarray} 
&&\frac{1}{h(D, d)} \sum_{ z \in \Lambda_d} f(z) \left( K^{\pm}(z,z) - M(X;z,z) \right) \nonumber \\
&&= \sum_{t_j \in \mathbb{R}} \frac{h_{\pm}(t_j)}{h(D,d)} \sum_{ z \in \Lambda_d} f(z) |u_j(z)|^2 + O_{f, \Gamma} (\delta X + X^{1/2} \log X). 
\end{eqnarray}
The last sum can be rewritten as
\begin{eqnarray} \label{splitoftheterm}
 \sum_{t_j \in \mathbb{R}} \frac{h_{\pm}(t_j)}{\vol(\mathcal{X}(D,1))} \int_{\mathcal{X}(D,1)} f(z) |u_j(z)|^2 d &+&   \sum_{t_j \in \mathbb{R}} h_{\pm}(t_j) G(d, t_j) \nonumber \\
&+& O(\delta X + X^{1/2} \log X).  
\end{eqnarray}
By Theorem \ref{theorembiro}, the first sum in (\ref{splitoftheterm}) is bounded by $X^{5/8+\epsilon}$. For the second sum, we first estimate $G(d,t_j)$ in dyadic intervals as follows. For real $t_j \in [T,2T]$, Proposition \ref{propositionboundsgdt} and bound (\ref{iwaniecsarnakbound}) implies
\begin{eqnarray} \label{interpolation1}
 \sum_{T \leq t_j \leq 2T} G(d,t_j)  \ll_{f, \Gamma, \epsilon}  |d|^{-1/12+\epsilon} T^{41/12 + \epsilon}.
\end{eqnarray}
Secondly, interchanging the sums and using inequality (\ref{localweylslawcorollary}) we get
\begin{eqnarray*}
 \sum_{t_j \in \mathbb{R}} \frac{1}{h(D,d)} \sum_{ z \in \Lambda_d} f(z) |u_j(z)|^2 \ll_{f, \Gamma} T^{2}
\end{eqnarray*}
hence
\begin{eqnarray} \label{interpolation2}
 \sum_{T \leq t_j \leq 2T} G(d,t_j)  \ll_{f, \Gamma} T^{2}.
\end{eqnarray}
Interpolation between estimates (\ref{interpolation1}), (\ref{interpolation2}) gives
\begin{eqnarray} \label{interpolationbound}
\sum_{T \leq t_j \leq 2T} G(d, t_j) \ll_{f, \epsilon} T^{2-2a} |d|^{-\frac{a}{12} +\epsilon} T^{\frac{41}{12} a +\epsilon} = |d|^{-\frac{a}{12} +\epsilon} T^{2 + \frac{17}{12} a +\epsilon}
\end{eqnarray}
for any $a \in [0,1]$. We apply estimates for the S/H-C transform from part $c)$ of Proposition \ref{selberghasichandratransform}   and estimate (\ref{interpolationbound}). We get
\begin{eqnarray}
\sum_{t_j \in \mathbb{R}} h_{\pm}(t_j) G(d, t_j) &=& \sum_{|t_j| < \delta^{-1}} h_{\pm}(t_j) G(d, t_j) + \sum_{|t_j| \geq \delta^{1}} h_{\pm}(t_j) G(d, t_j) \nonumber \\
&\ll& X^{1/2} |d|^{-a/12 +\epsilon} \delta^{-1(1/2 + 17a/12 +\epsilon)}
\end{eqnarray}
under the condition $a < 12/17$.  Thus, the discrete average of the error term is balanced by
\begin{eqnarray}
 O \left( X^{5/8+\epsilon} +X^{1/2} |d|^{-a/12+\epsilon} \delta^{-1(1/2 + 17a/12 +\epsilon)} + \delta X \right).  
\end{eqnarray}
Picking $a$ close to $12/17$ and balancing $\delta = X^{-1/5} |d|^{-2/85}$ we complete the proof of the theorem.
\end{proof}

\end{document}